\documentclass[11pt]{amsart}
\usepackage[margin=1.3in]{geometry}
\usepackage[utf8]{inputenc}
\usepackage{amsmath}
\usepackage[pagebackref=true]{hyperref}
\usepackage{amstext}
\usepackage{amscd}
\usepackage{amsmath}
\usepackage{amsthm}
\usepackage{latexsym}
\usepackage{url}
\usepackage{amssymb}
\usepackage[all]{xy}
\usepackage{mathrsfs} 
\usepackage{color}
\usepackage{tikz}
\usepackage{indentfirst}
\usepackage{MnSymbol}
\usepackage{enumitem}

\newtheorem{theorem}{Theorem}[section]
\newtheorem{thmx}{Theorem} 
 
\newtheorem{lemma}[theorem]{Lemma}

\newtheorem{proposition}[theorem]{Proposition}
\newtheorem{conjecture}[theorem]{Conjecture}
\newtheorem*{theorem*}{Theorem}

\theoremstyle{definition}
\newtheorem{definition}[theorem]{Definition}
\newtheorem{theorem-definition}[theorem]{Theorem-Definition}

\theoremstyle{remark}
\newtheorem{remark}[theorem]{Remark}
\newtheorem{example}[theorem]{Example}
\newtheorem{question}[theorem]{Question}

\numberwithin{equation}{section}

\hypersetup{
     colorlinks,breaklinks,
     citecolor    = [RGB]{0,50,175},
     linkcolor =[RGB]{140,20,0}
}

\begin{document}
\title[Characterization of abelian varieties for log pairs]{On the characterization of abelian varieties for log pairs in zero and positive characteristic}
\author{Yuan Wang}
\subjclass[2010]{ 
14E15, 14J10, 14J17, 14J30.
}
\keywords{
abelian variety, Albanese map, Albanese morphism, canonical bundle formula, generic vanishing, Kodaira dimension, positive characteristic.
}
\address{Department of Mathematics, Northwestern University, 2033 Sheridan Road
Evanston, IL  60208, USA}
\email{yuanwang@math.northwestern.edu}
\thanks{The author was supported in part by the NSF research grant DMS-\#1300750 and the Simons Foundation Award \# 256202.}
\begin{abstract}
Let $(X,\Delta)$ be a pair. We study how the values of the log Kodaira dimension and log plurigenera relates to surjectivity and birationality of the Albanese map and the Albanese morphism of $X$ in both characteristic $0$ and characteristic $p > 0$. In particular, we generalize some well known results for smooth varieties in both zero and positive characteristic to varieties with various types of singularities. Moreover, we show that if $X$ is a normal projective threefold in characteristic $p>0$, the coefficients of the components of $\Delta$ are $\le 1$ and $-(K_X+\Delta)$ is semiample, then the Albanese morphism of $X$ is surjective under reasonable assumptions on $p$ and the singularities of the general fibers of the Albanese morphism. This is a positive characteristic analog in dimension 3 of a result of Zhang on a conjecture of Demailly-Peternell-Schneider.
\end{abstract}
\maketitle
\section{Introduction}
Over the past few decades, there have been many interesting discoveries on the birational geometry of the Albanese morphism in both zero and positive characteristic, especially in higher dimensions. The theorem \cite[Theorem 1]{Kawamata81} of Kawamata shows that for a smooth complex projective variety $X$, if the Kodaira dimension of $X$ is $0$ then the Albanese morphism of $X$ is an algebraic fiber space. Later, Chen and Hacon (\cite{CH01,CH02}) proved that, assuming the irregularity of $X$ is equal to the dimension of $X$, then as long as one of the plurigenera $H^0(X,\mathcal{O}_X(mK_X))$ is $1$ for some $m\ge 2$, $X$ is birational to an abelian variety. This solves a conjecture of Koll\'{a}r. More recently, Hacon, Patakfalvi and L. Zhang established some positive characteristic analogs of Kawamata's result for smooth varieties and for possibly singular varieties of maximal Albanese dimension (see \cite{HP13, HP15b, HPZ17}). More specifically, this result in positive characteristic shows that, for a smooth projective variety $X$, if the Frobenius stable Kodaira dimension $\kappa_S(X)$ is $0$, then the Albanese morphism of $X$ is surjective; moreover, if $\dim X=\dim ({\rm Alb}(X))$, then $X$ is birational to an abelian variety. It is then natural to ask whether such results hold for singular varieties as well. Therefore, we formulate this question in the language of pairs as follows:
\begin{question}\label{questionlogpairsurj}
For a pair $(X,\Delta)$ (possibly with some assumption on the singularities), how do the log Kodaira dimension $\kappa(K_X+\Delta)$, the log plurigenera $h^0(X,\mathcal{O}_X(m(K_X+\Delta)))$ as well as the irregularity $q(X)$ characterize abelian varieties?
\end{question}
Another motivation originates from the following conjecture of Demailly-Peternell-Schneider:
\begin{conjecture}\cite[Conjecture 2]{DPS93}\label{conjKahlersurj}
Let $X$ be a compact K\"{a}hler manifold with $-K_X$ nef. Then the Albanese morphism of $X$ is surjective.
\end{conjecture}
\noindent In \cite{Zhang96}, Q. Zhang confirms Conjecture \ref{conjKahlersurj} when $X$ is a smooth projective variety. Moreover he shows that in this case the Albanese morphism of $X$ has connected fibers. Later Q. Zhang generalized this result to the cases of log canonical pairs (\cite[Corollary 2]{Zhang05}). More specifically, he showed that if $X$ is a projective variety and $\Delta$ is an effective $\mathbb{Q}$-divisor such that the pair $(X,\Delta)$ is log canonical and $-(K_X+\Delta)$ is nef, then the Albanese morphism from any smooth model of $X$ is an algebraic fiber space.

From the characteristic $p$ perspective the above results by Q. Zhang induces the following natural question:
\begin{question}\label{Zhangcharp}
Does a statement similar to \cite[Corollary 2]{Zhang05} hold in positive characteristic?
\end{question}

In this paper, we answer a large part of Question \ref{questionlogpairsurj} in all dimensions and answer Question \ref{Zhangcharp} in dimension $3$. The first main result of this paper is as follows. It generalizes \cite[the main theorem]{CH01}, \cite[Theorem 1, Corollary 2]{CH02}, \cite[Theorem 1.1.1]{HP13} and \cite[Theorem 0.1]{HPZ17} to the cases of log pairs as follows.
\begin{thmx}[Theorem \ref{maincharp} and \ref{charactCHlog}] \label{HPgen}
Let $(X,\Delta)$ be a projective pair over an algebraically closed field $k$. Suppose that one of the following conditions holds:
\begin{enumerate}[label={\rm (\alph*)}]
\item ${\rm char}(k)=0$, $(X,\Delta)$ is klt, $h^0(X,\mathcal{O}_X(m(K_X+\Delta)))=1$ for some $m\ge 2$, and either $\dim (\mathfrak{Alb}(X))=\dim X$ or $\dim ({\rm Alb}(X))=\dim X$. \label{partCH}
\item ${\rm char}(k)=0$, $(X,\Delta)$ is log canonical, $h^0(X,\mathcal{O}_X(m(K_X+\Delta)))=h^0(X,\mathcal{O}_X(2m(K_X+\Delta)))=1$ for some $m\ge 1$, and either $\dim (\mathfrak{Alb}(X))=\dim X$ or $\dim ({\rm Alb}(X))=\dim X$. \label{partch01}
\item ${\rm char}(k)=p>0$, $\kappa_S(X+\Delta)=0$, the index of $K_X+\Delta$ is not divisible by $p$, and $\dim ({\rm Alb}(X))=\dim X$. \label{part3charp}
\end{enumerate}
Then $X$ is birational to an abelian variety.
\end{thmx}
Note that unlike the smooth case, in the log setting the converse of Theorem \ref{HPgen} can easily fail. For example we can take $X$ to be an abelian variety and $\Delta$ a sufficiently ample divisor with small coefficient. ${\rm Alb}(X)$ and $\mathfrak{Alb}(X)$ in Theorem \ref{HPgen} means the Albanese variety constructed from the Albanese morphism and the Albanese map respectively, and they are different in general for singular varieties (see Section \ref{prelim}). Theorem \ref{HPgen} is sharp in the sense that it does not work without the condition that $(X,\Delta)$ is lc, as is illustrated by Example \ref{surfnonsurj}.

To prove part \ref{part3charp}, we use the same techniques as in \cite[Theorem 1.1.1]{HP13} and \cite[Theorem 0.1]{HPZ17}. Nevertheless, part \ref{part3charp} is interesting in the sense that to the best of the author's knowledge, there is no result of this kind in characteristic $0$. 

In the course of establishing Theorem \ref{partCH}, we also prove the following result on the structure of the Albanese map and Albanese morphism, which might be of independent interest.
\begin{thmx}[Theorem \ref{charactnonklt}] \label{Notalgfibdom}
Let $X$ be a normal complex projective variety, $(X,\Delta)$ a pair such that $\kappa(K_X+\Delta)=0$. Suppose that the Albanese map $\mathfrak{a}_X:X\dasharrow \mathfrak{Alb}(X)$ (respectively the Albanese morphism $a_X:X\to {\rm Alb}(X)$) of $X$ is not an algebraic fiber space. Then the non-lc locus of $(X,\Delta)$ dominates $\mathfrak{a}_X(X)$ (respectively $a_X(X)$).
\end{thmx}
Theorem \ref{Notalgfibdom} generalizes \cite[Theorem 1]{Kawamata81} beyond the log canonical case. The proof of it has various kinds of ingredients, including analyzing in detail the structure of the Albanese map and the Albanese morphism, using Campana's work on the Iitaka conjecture (\cite{Campana04}) over a variety of general type, and applying a generic vanishing result developed by Popa and Schnell (\cite{PS14}) to compare certain Kodaira dimensions. 


We next present our positive characteristic analog of \cite[Corollary 2]{Zhang05} in dimension $3$ as follows. The proof of this result is inspired by some ideas in the previous paper \cite{Wang15b} of the author.
\begin{thmx}[Theorem \ref{surj3d}] \label{Zhangcharpanal}
Let $X$ be a normal projective threefold over an algebraically closed field $k$ of characteristic $p>0$, and $(X,\Delta)$ a pair such that the coefficients of the components of $\Delta$ are $\le 1$ and $-(K_X+\Delta)$ is semiample.  Suppose that 
\begin{enumerate}[label={\rm (\arabic*)}]
\item $\displaystyle p>\max\{\dfrac{2}{\delta}, i_{\rm bpf}(-(K_X+\Delta))\}$, where $\delta$ is the minimal non-zero coefficient of $\Delta$. \label{charbigcond}
\item There is a fiber $X_0$ of the Albanese morphism $a_X$ of $X$ such that $(X_0,\Delta|_{X_0})$ is normal and sharply $F$-pure. \label{sFp}
\end{enumerate}
Then $a_X$ is surjective.
\end{thmx}
Here $i_{\rm bpf}(-(K_X+\Delta))$ is the \emph{base-point-free index} of $-(K_X+\Delta)$, which is the smallest integer $m>0$ such that $-m(K_X+\Delta)$ is base point free (cf. Definition \ref{bpfindex}). In particular if we assume that $\Delta=0$ and $|-K_X|$ is base point free, then condition \ref{charbigcond} is automatically satisfied, and in that case, Theorem \ref{Zhangcharpanal} works in small characteristics as well.

It is worth noting that Theorem \ref{HPgen}, \ref{Notalgfibdom} as well as \cite[Corollary 2]{Zhang05} does not hold in general if $(X,\Delta)$ is not log canonical, even in dimension $2$, as is shown in Example \ref{surfnonsurj}. The example is a projective cone $S$ over a non-hyperelliptic curve of genus $\ge 2$ such that $K_S\sim_{\rm lin}0$, so obviously its Albanese map of $S$ is not surjective (but the Albanese morphism of $S$ is). Observe that in this case the vertex of $S$ is its non-lc center, and the exceptional divisor over the vertex dominates the image of the Albanese map of $S$, which is the image of the curve in its Jacobian. 
\subsection*{Acknowledgements} 
This project started from a question the author put forward in Christopher Hacon's class, and a large part of it was done when the author was at the University of Utah. The author would like to thank Christopher for his constant support, encouragement and many enlightening discussions. He is also indebted to Karl Schwede who provided lots of help on $F$-singularities, and Sofia Tirabassi who pointed out the very important reference \cite{Shibata15}. Moreover he thanks Omprokash Das, Tommaso de Fernex, Yoshinori Gongyo, Zsolt Patakfalvi, Mihnea Popa and Shunsuke Takagi for helpful conversations. Part of this project was done when the author was attending the conference  ``Higher Dimensional Algebraic Geometry and Characteristic $p$" at the Centre International de Rencontres Math\'{e}matiques and was visiting the University of Oslo, and he wishes to thank the organizers, especially John Christian Ottem, for their support and hospitality. 
\section{Notations and preliminaries}\label{prelim}
Throughout the article, we work over an algebraically closed field $k$. As we will consider the cases in both characteristic $0$ and characteristic $p>0$, we do not make any restriction on the characteristic of $k$ in advance.

We set up the following conventions. 
\begin{enumerate}
\item For definitions related to singularities in birational geometry (e.g. \emph{pairs}, \emph{discrepancy}, \emph{klt}, \emph{lc}) we follow \cite[Part I, Chapter 3]{HK10}. 
\item For a morphism $f:X\to Y$ of varieties we use both $F_{X/Y}$ and $F_f$ to denote a general fiber of $f$, and we say that $f$ is an \emph{algebraic fiber space} if it is surjective and has connected fibers. 
\end{enumerate}
\subsection{Albanese variety, Albanese map, Albanese morphism, derived categories, Fourier-Mukai transform and generic vanishing}
We first present the definitions for the Albanese map and the Albanese morphism. These can be found in \cite{Lang83} and \cite{FGIKNV05}.
\begin{theorem-definition}
Let $X$ be a variety over $k$. There exists an abelian variety $\mathfrak{Alb}(X)$ together with a rational map $\mathfrak{a}_X:X\dasharrow\mathfrak{Alb}(X)$ such that
\begin{itemize}
\item $\mathfrak{a}_X(X)$ generates $\mathfrak{Alb}(X)$, i.e. $\mathfrak{a}_X(X)$ is not contained in a translate of any proper abelian subvariety of $\mathfrak{Alb}(X)$.
\item For every rational map $f: X\dasharrow A$ from $X$ to an abelian variety $A$, there exists a homomorphism $g:\mathfrak{Alb}(X)\to A$ and a constant $a\in A$ such that $f=g\circ \mathfrak{a}_X+a$.
\end{itemize}
We call $\mathfrak{a}_X$ the \emph{Albanese map} of $X$ and $\mathfrak{Alb}(X)$ the \emph{Albanese variety via the Albanese map} of $X$.
\end{theorem-definition}
\begin{theorem-definition}
Let $X$ be a normal projective variety over $k$. There exists an abelian variety ${\rm Alb}(X)$ together with a morphism $a_X:X\to{\rm Alb}(X)$ such that
\begin{itemize}
\item $a_X(X)$ generates ${\rm Alb}(X)$, i.e. $a_X(X)$ is not contained in a translate of any proper abelian subvariety of ${\rm Alb}(X)$.
\item For every morphism $f: X\to A$ from $X$ to an abelian variety $A$, there exists a homomorphism $h:{\rm Alb}(X)\to A$ and a constant $b\in A$ such that $f=h\circ a_X+b$.
\end{itemize}
We call $a_X$ the \emph{Albanese morphism} of $X$ and ${\rm Alb}(X)$ the \emph{Albanese variety via the Albanese morphism} of $X$.
\end{theorem-definition}

Throughout the article, unless otherwise stated, for a normal projective variety $X$ we use $\mathfrak{a}_X:X\dasharrow\mathfrak{Alb}(X)$ ($a_X:X\to{\rm Alb}(X)$) to denote the Albanese map (Albanese moprhism) of $X$. We refer to \cite[Chapter II, \S 3]{Lang83} and \cite[Chapter 9]{FGIKNV05} for more details about the Albanese map and the Albanese morphism respectively. The Albanese map and the Albanese morphism agree in characteristic $0$ for normal proper varieties with rational singularities (see \cite[Proposition 2.3]{Reid83} or \cite[Lemma 8.1]{Kawamata85}). But they do not agree in general, as is illustrated by the following
\begin{example}\label{mamodiff}
Let $X$ be a projective cone over any curve $C$ of genus $\ge 1$. Since $X$ is covered by rational curves passing through the vertex and the Albanese morphism contracts all the rational curves, we see that the Albanese morphism of $X$ has to be a morphism from $X$ to a point. On the other hand, let $X'$ be the blow-up of $X$ at the vertex, $\nu:X\dasharrow X'$ the natural birational map, $p:X'\to C$ the natural $\mathbb{P}^1$ fibration from $X'$ to $C$ and $j:C\hookrightarrow {\rm Jac(C)}$ the embedding from $C$ to its Jacobian. Then the Albanese map of $X$ is $j\circ p\circ \nu$, whereas the Albanese morphism (and also the Albanese map) of $X'$ is $j\circ p$.
\end{example}
\begin{remark} \label{birsurAlb}
Example \ref{mamodiff} is also an example of the fact that the Albanese variety via the Albanese map is a birational invariant, but the Albanese variety via the Albanese morphism is not. Still, by definition, for a normal projective variety $X$, there exists a surjective homomorphism $\alpha$ from $\mathfrak{Alb}(X)$ to ${\rm Alb}(X)$ such that $a_X=\alpha\circ\mathfrak{a}_X$.
\end{remark}

For convenience, in characteristic $0$ we say that \emph{the Albanese map of $X$ is an algebraic fiber space} if the Albanese morphism of any smooth model of $X$ is an algebraic fiber space.
\begin{definition}\label{cohomsupploci}
For an abelian variety $A$ and a coherent sheaf $\mathcal{F}$ on $A$, we define 
$V^i(\mathcal{F})$, the \emph{$i$-th cohomology support locus}, as $$V^i(\mathcal{F})=\{P\in {\rm Pic}^0(A)|h^i(A,\mathcal{F}\otimes P)> 0\}.$$ $\mathcal{F}$ is called a \emph{GV-sheaf}, if codim$_{{\rm Pic}^0(A)}V^i(\mathcal{F})\ge i$ for all $i\ge 0$. 
\end{definition}
\begin{theorem}\cite[Corollary 3.2 (1)]{Hacon04} \label{Hac04Vi}
Let $\mathcal{F}$ be a GV-sheaf on an abelian variety $A$. For any $P\in {\rm Pic}^0(A)$, if $i\ge 0$ and $H^i(A,\mathcal{F}\otimes P)=0$, then $H^{i+1}(A,\mathcal{F}\otimes P)=0$.
\end{theorem}
\begin{theorem}\cite[Variant 5.5]{PS14}\label{ps14}
Let $f: X\to A$ be a morphism from a normal complex projective variety to an abelian variety. If $(X,\Delta)$ is an lc pair and $k>0$ is any integer such that $k(K_X+\Delta)$ is Cartier, then $f_*\mathcal{O}_X(k(K_X+\Delta))$ is a GV-sheaf. 
\end{theorem}
For a normal projective variety of dimension $n$, we denote by $D(X)$ the derived category of $\mathcal{O}_X$ modules, and denote by $D_{\rm qc}(X)$ the full subcategory of $D(X)$ consisting of bounded complexes with quasi-coherent cohomologies. We denote the dualizing complex of $X$ by $\omega_X^{\cdot}$ and define the dualizing functor $D_X$ by $D_X(F):=R\mathcal{H}om(F,\omega_X^{\cdot}[n])$ for any $F\in D_{\rm qc}(X)$.
\begin{definition}\cite{Mukai81}
Let $A$ be an abelian variety and $\hat{A}$ its dual abelian variety. Let $p_A: A\times \hat{A}\to A$ and $p_{\hat{A}}:A\times \hat{A}\to \hat{A}$ be the natural projection morphisms and $\mathcal{P}$ the Poincar\'{e} line bundle on $A\times \hat{A}$. We define the \emph{Fourier-Mukai transform} $R\hat{S}:D_{\rm qc}(A)\to D_{\rm qc}(\hat{A})$ and $RS:D_{\rm qc}(\hat{A})\to D_{\rm qc}(A)$ with respect to the kernel $\mathcal{P}$ by
$$R\hat{S}(\cdot)=Rp_{\hat{A},*}(p_A^*(\cdot)\otimes\mathcal{P}),\ \ \ RS(\cdot)=Rp_{A,*}(p_{\hat{A}}^*(\cdot)\otimes\mathcal{P}).$$
\end{definition}
For an abelian variety $A$ of dimension $g$, $R^g\hat{S}$ gives an equivalence in categories between unipotent vector bundles on $A$ and the category of Artinian $\mathcal{O}_{\hat{A},\hat{0}}$-modules of finite length on $\hat{A}$ (see \cite[Example 2.9]{Mukai81} and \cite[Section 1.2]{CH09}).
\begin{lemma} \label{GVandKodairadim}
Suppose that ${\rm char}(k)=0$. Let $(X,\Delta)$ be a log canonical pair and $f:X\to A$ an algebraic fiber space from $X$ to an abelian variety $A$. If $\kappa(K_{F_{f}}+\Delta|_{F_{f}})\ge 0$, then $\kappa(K_X+\Delta)\ge 0$.
\end{lemma}
\begin{proof}
If $\kappa(K_{F_{f}}+\Delta|_{F_{f}})\ge 0$, then $H^0(F_{f}, \mathcal{O}_{F_{f}}(m(K_{F_f}+\Delta|_{F_f})))\ne 0$ for a sufficiently divisible $m$ such that $m(K_X+\Delta)$ is Cartier. So by cohomology and base change we have that $f_*(\mathcal{O}_{X}(m(K_X+\Delta)))\ne 0$. 

Next we prove that $V^0(f_*\mathcal{O}_X(m(K_X+\Delta)))\ne 0$. By Theorem \ref{ps14} we know that $f_*\mathcal{O}_X(m(K_X+\Delta))$ is a GV sheaf, so by Theorem \ref{Hac04Vi},
$$V^i(f_*\mathcal{O}_X(m(K_X+\Delta)))\supseteq V^{i+1}(f_*\mathcal{O}_X(m(K_X+\Delta)))$$
for all $i\ge 0$. Suppose that $V^0(f_*\mathcal{O}_X(m(K_X+\Delta)))=0$. Then $V^i(f_*\mathcal{O}_X(m(K_X+\Delta)))=0$ for all $i$. By cohomology and base change we have $R\hat{S}(f_*\mathcal{O}_X(m(K_X+\Delta)))=0$ (cf. \cite[Proof of Proposition 2.13]{Wang15}). So by \cite[Theorem 2.2]{Mukai81} we get $f_*(\mathcal{O}_{X}(m(K_X+\Delta)))=0$,  which contradicts what we deduced in the last paragraph.

Now we have that $V^0(f_*\mathcal{O}_X(m(K_X+\Delta)))\ne 0$. On the other hand, by \cite[Theorem 1.3]{Shibata15}, $V^0(f_*\mathcal{O}_X(m(K_X+\Delta)))$ is a finite union of torsion translates of abelian subvarieties of ${\rm Pic}^0(A)$. In particular, 
$$0\ne H^0(A, f_*\mathcal{O}_X(m(K_X+\Delta))\otimes P)=H^0(X, \mathcal{O}_X(m(K_X+\Delta))\otimes f^*P)$$
for some torsion elements $P\in {\rm Pic}^0(A)$. This implies that $H^0(X, \mathcal{O}_X(m'(K_X+\Delta)))\ne 0$ for sufficiently divisible $m'$, hence $\kappa(K_X+\Delta)\ge 0$.
\end{proof}
\subsection{$F$-singularities}
We recall the definitions of some basic concepts of $F$-singularity. We refer to \cite{ST11}, \cite{PST14}, \cite{HX15} and \cite{DS15} for more details. Let $k$ be of characteristic $p>0$ and $X$ a normal projective variety over $k$. Let $(X,\Delta)$ be a pair such that the index of $K_X+\Delta$ is not divisible by $p$. Let
$$\mathcal{L}_{e,\Delta}:=\mathcal{O}_X((1-p^{e})(K_X+\Delta))$$
where $e$ is a positive integer where $e$ is a positive integer such that $(p^{e}-1)(K_X+\Delta)$ is integral. Then there is a canonically determined morphism $\phi_e:F_*^e\mathcal{L}_{e,\Delta}\to \mathcal{O}_X$ (see \cite[Section 2.3]{HX15}). We define the \emph{non-$F$-pure ideal} $\sigma(X,\Delta)$ (the \emph{test ideal} $\tau(X,\Delta)$) to be the unique biggest (smallest) ideal $J\subseteq\mathcal{O}_X$ such that $(\phi_e\circ F_X^e)(J\cdot\mathcal{L}_{e,\Delta})=J$ for any $e$. We say that $(X,\Delta)$ is \emph{sharply $F$-pure} (\emph{strongly $F$-regular}) if $\sigma(X,\Delta)=\mathcal{O}_X$ ($\tau(X,\Delta)=\mathcal{O}_X$). For an integral subscheme $Z\subseteq X$ with generic point $\eta_Z$, we say that $Z$ is an \emph{$F$-pure center} of $(X,\Delta)$ if
\begin{itemize}
\item $\Delta$ is effective at $\eta_Z$,
\item $(X,\Delta)$ is $F$-pure at $\eta_Z$, and
\item $\phi(F_*^eI_Z\cdot \mathcal{L}_{e,\Delta})\cdot\mathcal{O}_{X,\eta_Z}\subseteq I_Z\cdot\mathcal{O}_{X,\eta_Z}$.
\end{itemize}
\begin{lemma}\label{Fpureperturb}
Let $(X,\Delta)$ be a sharply $F$-pure pair such that the index of $(K_X+\Delta)$ is not divisible by $p$. Let $D$ be a divisor such that $D$ does not contain any $F$-pure center of $(X,\Delta)$. Then there is a positive integer $n$ such that $(X,\Delta+\dfrac{1}{n}D)$ is sharply $F$-pure.
\end{lemma}
\begin{proof}
We first construct a stratification for the $F$-pure centers of $(X,\Delta)$. Let $W_i^*(X,\Delta)$ be the union of $\le i$-dimensional $F$-pure centers of $(X,\Delta)$. Let
$$W_i(X,\Delta):=W_i^*(X,\Delta)\setminus W_{i-1}^*(X,\Delta)$$ 
and $W_i^k(X,\Delta)$ the irreducible components of $W_i(X,\Delta)$. Then $X$ is a disjoint union of $W_i^k(X,\Delta)$, and each $W_i^k(X,\Delta)$ is a minimal $F$-pure center of $(X,\Delta)$. In particular $W_i^k(X,\Delta)$ is normal. For each $W_i^k(X,\Delta)$, by \cite[(i) and (iv) of Main Theorem]{Schwede09} there is an effective divisor $\Delta_{W_i^k(X,\Delta)}$ on $W_i^k(X,\Delta)$ such that $(K_X+\Delta)|_{W_i^k(X,\Delta)}=K_{W_i^k(X,\Delta)}+\Delta_{W_i^k(X,\Delta)}$ and $(W_i^k(X,\Delta), \Delta_{W_i^k(X,\Delta)})$ is strongly $F$-regular. Then there is a positive integer $n_i^k>0$ such that $(W_i^k(X,\Delta), \Delta_{W_i^k(X,\Delta)}+\dfrac{1}{n_i^k}D|_{W_i^k(X,\Delta)})$ is strongly $F$-regular. Hence by \cite[(iii) of Main Theorem]{Schwede09}, $(X,\Delta+\dfrac{1}{n_i^k}D)$ is sharply $F$-pure near $W_i^k(X,\Delta)$. Finally let $n:={\rm max}\{n_i^k\}$ and we are done.
\end{proof}

Let $f:X\to Y$ be a morphism between normal projective varieties and $\mathcal{M}$ a line bundle on $X$. We define the subsheaf $S^0f_*(\sigma(X,\Delta)\otimes \mathcal{M})\subseteq f_*\mathcal{M}$ as
\begin{align}\label{S0definition}
\begin{split}
& S^0f_*(\sigma(X,\Delta)\otimes\mathcal{M}):=\bigcap_{g\ge 0}{\rm Image}(F_*^{ge}f_*(\mathcal{L}_{ge,\Delta}\otimes\mathcal{M}^{p^{ge}}) \xrightarrow{\Phi_{ge}} f_*\mathcal{M}) \\
\end{split}
\end{align}
where $\Phi_{ge}$ is induced by $\phi_{ge}$ using the projection formula and the commutativity of $F$ and $f$. In the case $Y={\rm Spec}\ k$, $f_*(\cdot)$ is the same as $H^0(\cdot)$, and in this case we use $S^0(X, \Delta; \mathcal{M})$ to denote $S^0f_*(\sigma(X,\Delta)\otimes \mathcal{M})$. Since $H^0(X,\mathcal{M})$ is finite dimensional, we have
\begin{align*}
S^0(X, \Delta; \mathcal{M})={\rm Image}(H^0(X, F_*^{me}(\mathcal{L}_{me,\Delta}\otimes\mathcal{M}^{p^{me}}))\xrightarrow{\Phi_{me}} H^0(X, \mathcal{M}))
\end{align*}
for some $m\gg 0$. 

The following definition is analogous to that in \cite[4.1]{HP13}.
\begin{definition}
Let $(X,\Delta)$ be a pair and $r>0$ an integer such that $r(K_X+\Delta)$ is Cartier. The \emph{Frobenius stable Kodaira dimension} of $(X,\Delta)$ is defined as
\begin{align*}
&\kappa_S(K_X+\Delta) \\
&:={\rm max}\{k|\dim S^0(X, \Delta; \mathcal{O}_X(mr(K_X+\Delta)))=O(m^k)\ {\rm for}\ m\ {\rm sufficiently\ divisible}\}.
\end{align*}
\end{definition}
It is easy to see that $\kappa_S(K_X+\Delta)$ is independent of $r$.
\subsection{A canonical bundle formula in characteristic $p$}
The following theorem is a slight modification of \cite[Theorem 4.3]{Wang15b}, which is a canonical bundle formula for morphism from a threefold to a surface whose general fibers are $\mathbb{P}^1$. The same proof in \cite{Wang15b} works.
\begin{theorem}\label{canonbund}
Let $f:X\to Y$ be a proper surjective morphism, where $X$ is a normal threefold and $Y$ is a normal surface over an algebraically closed field $k$ of characteristic $p>0$. Assume that $Q=\sum_i Q_i$ is a divisor on $Y$ such that $f$ is smooth over $(Y-{\rm Supp}(Q))$ with fibers isomorphic to $\mathbb{P}^1$. Let $D$ be an effective $\mathbb{Q}$-divisor on $X$, which satisfies the following conditions: 
\begin{enumerate}
\item $(X,D)$ is lc on a general fiber of $f$. 
\item Suppose $D=D^h+D^v$ where $D^h$ is the horizontal part and $D^v$ is the vertical part of $D$. Then $p={\rm char}(k)>\dfrac{2}{\delta}$, where $\delta$ is the minimum non-zero coefficient of $D^h$.
\item $K_X+D\sim_{\mathbb{Q}}f^*(K_Y+M)$ for some $\mathbb{Q}$-Cartier divisor $M$ on $Y$.
\end{enumerate}
Then $M$ is $\mathbb{Q}$-linearly equivalent to an effective $\mathbb{Q}$-divisor. 
\end{theorem}
\section{Results in characteristic $p>0$}
In this section we first prove the following result about the surjectivity of Albanese morphism in positive characteristic.
\begin{proposition}\label{surj}
Let $(X,\Delta)$ be a projective pair in characteristic $p>0$. Assume that $\kappa_S(K_X+\Delta)=0$ and the index of $K_X+\Delta$ is not divisible by $p$. Then the Albanese morphism of $X$ is surjective.
\end{proposition}
\begin{proof}
For convenience, in this proof we use $a:X\to A$ to denote the Albanese morphism of $X$. $\kappa_S(K_X+\Delta)=0$ implies that there is a positive integer $r$ such that $r(K_X+\Delta)$ is integral and $S^0(X,\Delta;\mathcal{O}_X(r(K_X+\Delta)))\ne 0$. We fix an $e>0$ such that $p^e-1$ is divisible by the index of $K_X+\Delta$. Then for an $m\gg 0$ there is an element of $H^0(X,\mathcal{O}_X((p^{me}r+(1-p^{me}))(K_X+\Delta)))$ mapping to $0\ne f\in H^0(X,\mathcal{O}_X(r(K_X+\Delta)))$. By \cite[Lemma 4.1.3]{HP13} we know that $\kappa(K_X+\Delta)=0$, so that element can only be $\alpha f^{p^{me}+\frac{1-p^{me}}{r}}$ for some $\alpha\in k^*$. 
Let $G\in|r(K_X+\Delta)|$ be the unique element and $D:=\frac{r-1}{r}G$. Since the image of 
$$\alpha f\in H^0(X,\mathcal{O}_X(r(K_X+\Delta)))=H^0(X,\mathcal{O}_X(rp^{me}(K_X+\Delta)+(1-p^{me})(K_X+\Delta+D)))$$
in $H^0(X,\mathcal{O}_X(r(K_X+\Delta)))$ is computed by
$$\alpha f\mapsto \alpha f\cdot f^{\frac{r-1}{r}(p^{me}-1)}=\alpha f^{p^{me}+\frac{1-p^{me}}{r}}\mapsto\beta f,$$
we have actually proved that $S^0(X,\Delta+D; \mathcal{O}_X(r(K_X+\Delta)))\ne 0$. Then the natural inclusion (see \cite[Lemma 2.2.2]{HP13})
$$S^0(X,\Delta+D; \mathcal{O}_X(r(K_X+\Delta)))\subseteq H^0(A,\Omega_0)\subseteq H^0(X,\mathcal{O}_X(r(K_X+\Delta)))$$
are equalities, where $\Omega_0=S^0a_*(\sigma(X,\Delta+D)\otimes\mathcal{O}_X(r(K_X+\Delta)))$. We consider the following diagram
\begin{center}
\begin{tikzpicture}[scale=1.6]
\node (A) at (0,0) {$F_*^{(m+1)e}a_*\mathcal{O}_X(r(K_X+\Delta))$};
\node (B) at (4,0) {$F_*^{me}a_*\mathcal{O}_X(r(K_X+\Delta))$};
\node (C) at (0,1) {$F_*^{(m+1)e}\Omega_0$};
\node (D) at (4,1) {$F_*^{me}\Omega_0$};  
\node (E) at (0,-1) {$a_*F_*^{(m+1)e}\mathcal{O}_X(r(K_X+\Delta))$};
\node (F) at (4,-1) {$a_*F_*^{me}\mathcal{O}_X(r(K_X+\Delta))$};
\path[->,font=\scriptsize]
(A) edge node[above]{} (B)
(C) edge node[above]{} (D);
\draw[->,font=\scriptsize]
(C) edge node[right]{} (A)
(D) edge node[right]{} (B)
(A) edge node[right]{$=$} (E)
(B) edge node[right]{$=$} (F)
(E) edge node[above]{} (F);
\end{tikzpicture} 
\end{center}  
where the bottom two horizontal arrows are from \cite[Lemma 2.6]{Patakfalvi13}. After applying $H^0(A,\cdot)$ we get that the bottom row is an isomorphism as the stable image of these maps is exactly $S^0(X,\Delta+D;\mathcal{O}_X(r(K_X+\Delta)))=H^0(X,\mathcal{O}_X(r(K_X+\Delta)))$, which is proved above. Therefore the natural maps 
$$H^0(A,F_*^{(m+1)e}\Omega_0)\to H^0(A,F_*^{me}\Omega_0)$$
and 
$$H^0(X,F_*^{(m+1)e}\mathcal{O}_X(r(K_X+\Delta)))\to H^0(X,F_*^{me}\mathcal{O}_X(r(K_X+\Delta)))$$ 
are isomorphisms.

Now by assumption there is a unique $G\in |r(K_X+\Delta)|$. We use the following notation
$$\Omega_{m}:=F_*^{me}S^0a_*(\sigma(X,\Delta+D)\otimes\mathcal{O}_X(r(K_X+\Delta)))\ {\rm and}\ \Omega=\varprojlim\Omega_{m}.$$
It is easy to see that $r|(p^{me}-1)$ for every $m\ge 0$. According to \cite[Lemma 2.6]{Patakfalvi13} and \cite[Lemma 8.1.4]{BS13} we know that $\Omega_{(m+1)}\to\Omega_{m}$ is surjective for every $m\ge 0$. Further by the above argument, $S^0a_*(\sigma(X,\Delta+D)\otimes\mathcal{O}_X(r(K_X+\Delta)))\ne 0$. So for every $m$, $\Omega_{m}\ne 0$, hence $\Omega\ne 0$.

Next we claim that since $\kappa(K_X+\Delta)=0$, there is a neighborhoold $U$ of the origin of $\hat{A}$ such that 
$$V^0(\mathcal{O}_X(r(K_X+\Delta)))\cap U={0_{\hat{A}}}.$$
Let $g:=\dim A$. Suppose that this is not the case and let $T\in \hat{A}$ be a positive dimensional irreducible component of $V^0(\mathcal{O}_X(r(K_X+\Delta)))$ through the origin. Let $\xi:T^{g+1}\to \hat{A}$ be the natural morphism. For dimensional reason, every fiber of $\xi$ is of positive dimension. Moreover since $0\in T$, $0$ is in the image of $\xi$. Consider for $(P_1,...,P_{g+1})\in\xi^{-1}(\mathcal{O}_{A})$ the maps
\begin{align*}
& H^0(X,\mathcal{O}_X(r(K_X+\Delta))\otimes a^*P_1)\otimes...\otimes H^0(X,\mathcal{O}_X(r(K_X+\Delta))\otimes a^*P_{g+1}) \\
& \to H^0(X,\mathcal{O}_X((g+1)r(K_X+\Delta))).
\end{align*}
Since $h^0(X,\mathcal{O}_X((g+1)r(K_X+\Delta)))=1$, there are only finitely many choices for the divisors in $|r(K_X+\Delta)+a^*P_i|$ for $P_i\in p_i(\xi^{-1}(\mathcal{O}_{A}))$. However there exists an $i$ such that $p_i(\xi^{-1}(\mathcal{O}_{A}))$ is an infinite set, and we also have $({\rm Pic}^0(X))_{\rm red}={\rm Pic}^0(A)=\hat{A}$ (see \cite[Proposition A.6]{Mochizuki12}). In particular the map ${\rm Pic}^0(A)\to{\rm Pic^0(X)}$ is injective. So we get a contradiction, hence the claim holds.

Since $H^0(A,S^0a_{X,*}(\sigma(X,\Delta+D)\otimes\mathcal{O}_X(r(K_X+\Delta)))\otimes P)\subseteq H^0(X,\mathcal{O}_X(r(K_X+\Delta))\otimes a^*P)$ we have that $H^0(A,\Omega_0\otimes P)$ is zero for every $\mathcal{O}_A\ne P\in U$. Moreover $H^0(A,\Omega_0)\ne 0$ as shown above, so by \cite[Corollary 3.2.1]{HP13} we have that $\mathcal{H}^0(\Lambda_0)$ is Artinian
in a neighborhood of the origin and $\mathcal{H}^0(\Lambda_0)\otimes k(0)\ne 0$, where 
$$\Lambda_m:=R\hat{S}(D_A(\Omega_m)).$$ 
Therefore we can write $\mathcal{H}^0(\Lambda_0)$ as $\mathcal{H}^0(\Lambda_0)=\mathcal{C}_0\oplus\mathcal{B}_0$, where $\mathcal{B}_0$ is Artinian and supported at $0$ and $0\notin{\rm Supp}\ \mathcal{C}_0$. This induces a similar decomposition $\mathcal{C}_m\oplus\mathcal{B}_m$ of $\mathcal{H}^0(\Lambda_m)\cong V^*\mathcal{H}^0(\Lambda_0)$, where $V$ is the Verschiebung isogeny (see \cite[(5.18) Definition]{MvdG}). Furthermore, the morphism $\mathcal{H}^0(\Lambda_m)\to\mathcal{H}^0(\Lambda_{m+1})$ is the direct sum of morphisms $\mathcal{C}_m\to\mathcal{C}_{m+1}$ and $\mathcal{B}_m\to\mathcal{B}_{m+1}$. Hence
$$\Lambda:=\varinjlim\Lambda_m=(\varinjlim\mathcal{C}_m)\oplus(\varinjlim\mathcal{B}_m),$$
where $\mathcal{B}_m$ is Artinian for every $m\ge 0$.

Now by \cite[Corollary 3.2.1]{HP13} we have $\mathcal{H}_m\otimes k(0)\cong \mathcal{B}_m\otimes k(0)\cong H^0(A,F_*^m\Omega_0)^{\vee}$, so the morphism $\mathcal{B}_m\otimes k(0)\to\mathcal{B}_{m+1}\otimes k(0)$ can be identified with $H^0(A,F_*^m\Omega_0)^{\vee}\to H^0(A,F_*^{m+1}\Omega_0)^{\vee}$. However by the above argument the latter morphism is an isomorphism, and $H^0(A,\Omega_0)\ne 0$. Therefore we know that $\varinjlim(\mathcal{B}_m\otimes k(0))$, hence $\varinjlim\mathcal{B}_m$, is nonzero. Finally by \cite[Corollary 3.2.3]{HP13} we are done.
\end{proof}
Proposition \ref{surj} then implies that
\begin{theorem}\label{maincharp}
Let $(X,\Delta)$ be a projective pair in characteristic $p>0$.  Assume that $\kappa_S(X+\Delta)=0$, $\dim({\rm Alb}(X))=\dim X$ and the index of $K_X+\Delta$ is not divisible by $p$. Then $X$ is birational to an ordinary abelian variety.
\end{theorem}
\begin{proof}
By Proposition \ref{surj}, $a_X$ is surjective, hence generically finite. By \cite[Proposition 1.4]{HPZ17}, $a_X$ is separable. So we have $K_X=a_X^*K_{{\rm Alb}(X)}+R$, where $R$ is an effective $\mathbb{Q}$-divisor supported on the ramification locus of $a_X$. This implies that $0=\kappa_S(K_X+\Delta)=\kappa(K_X+\Delta)\ge \kappa(K_X)\ge \kappa(K_{{\rm Alb}(X)})=0$ (note that we have proved $\kappa(K_X+\Delta)=\kappa_S(K_X+\Delta)=0$ in the proof of Proposition \ref{surj}), which forces that $\kappa(X)=0$, and by \cite[Theorem 0.1 (2)]{HP15b} we are done.
\end{proof}
The second main theorem in this section is a characteristic $p$ analog of \cite[Corollary 2]{Zhang05} in dimension 3. Before we present that we would like to point out that \cite[Corollary 2]{Zhang05} does not hold without the assumption that $(X,D)$ is lc. This is illustrated by the following
\begin{example}\label{surfnonsurj}
Let $C$ be a smooth non-hyperelliptic curve of genus $g\ge 2$. By \cite[Chapter IV, Proposition 5.2]{Hartshorne77}, $K_C$ is very ample. Let $C\hookrightarrow\mathbb{P}^r$ be the embedding induced by $|K_C|$. Let $S$ be the projective cone over $C$, $\nu: S'\to S$ the blow-up at the vertex of $S$, $\pi:S'\to C$ the induced morphism and $E$ the exceptional divisor of $\nu$. We then have that $E^2=-\deg K_C=2-2g$, and we can write
$$\nu^*K_S=K_{S'}+\lambda E$$
for some $\lambda$. By adjunction formula we see that
\begin{align*}
2g-2=(K_{S'}+E)\cdot E=(1-\lambda)E^2=(\lambda-1)(2g-2),
\end{align*}
so $\lambda=2$. If we denote a fiber of $\pi$ by $f$, then
\begin{align*}
& h^0(S,\mathcal{O}_S(K_S))=h^0(S',\nu^*\mathcal{O}_{S}(K_{S}))=h^0(S',\mathcal{O}_{S'}(K_{S'}+2E)) \\
& =h^0(S',\mathcal{O}_{S'}(-2E+(2g-2-\deg K_C)f+2E))=h^0(S',\mathcal{O}_{S'})=1,
\end{align*}
where the third equality is implied by the fact that $K_{S'}\sim_{\rm lin}-2E+(2g-2-\deg K_C)f$ as $K_C$ is very ample. In the above calculation we actually get that $\nu^*\mathcal{O}_{S}(K_{S})=\mathcal{O}_{S'}$, Hence $K_S\sim_{\rm lin}0$. In particular $-K_S$ is nef and $\kappa(S)=0$. On the other hand, by the construction of $S$ we know that its Albanese map is not surjective. When $g=2$, $(S', 2E)$ provides a counter-example to Theorem \ref{HPgen} and \ref{Notalgfibdom} without the log canonical assumption. We also note that by the argument in Example \ref{mamodiff}, the Albanese morphism of $S$ is the morphism from $S$ to a point. In particular it is an algebraic fiber space, hence surjective.
\end{example}
We now show the second main theorem. For convenience we give the following definition.
\begin{definition}\label{bpfindex}
For a $\mathbb{Q}$-Cartier divisor $D$ on a projective variety $X$, we define the \emph{base-point-free index} of $D$ to be the smallest number $m>0$ such that $|mD|$ is base point free and denote it as $i_{\rm bpf}(D)$.
\end{definition}
\begin{theorem}\label{surj3d}
Let $X$ be a normal projective threefold over an algebraically closed field $k$ of characteristic $p>0$, and $(X,\Delta)$ a pair such that the coefficients of the components of $\Delta$ are $\le 1$ and $-(K_X+\Delta)$ is semiample.  Suppose that 
\begin{enumerate}[label={\rm (\arabic*)}]
\item $\displaystyle p>\max\{\dfrac{2}{\delta}, i_{\rm bpf}(-(K_X+\Delta))\}$, where $\delta$ is the minimal non-zero coefficient of $\Delta$. 
\item There is a fiber $X_0$ of the Albanese morphism $a_X$ of $X$ such that $X_0$ is normal and $(X_0,\Delta|_{X_0})$ is sharply $F$-pure. \label{sFp}
\end{enumerate}
Then $a_X$ is surjective.
\end{theorem}
\begin{lemma}\label{surjsurface}
Let $X$ be a normal projective surface. Suppose that $X$ has at most log canonical singularities and $\kappa(X)=0$. Then the Albanese map (hence the Albanese morphism) of $X$ is surjective.
\end{lemma}
\begin{proof}
Let $\mu:X'\to X$ be a minimal resolution of singularities for $X$, then we have $\kappa(X')\le\kappa(X)$. 

If $\kappa(X')=0$ then this is straight forward by the Enriques-Kodaira classification. 

If $\kappa(X')=-\infty$ then there is a birational morphism $g:X'\to X''$ where $X''$ is a ruled surface over a curve $C$ with a $\mathbb{P}^1$ fibration $\pi:X''\to C$. We can write 
$$K_{X'}=\mu^*K_X-\sum_ia_iE_i=g^*K_{X''}+\sum_{j}b_jF_j$$ 
where $a_i, b_j>0$ and $E_i$ and $F_j$ are exceptional curves on $X'$ for $\mu$ and $g$ respectively. By assumption we have 
\begin{align*}
0\ne H^0(X, \mathcal{O}_X(mK_X)) & =H^0(X', \mu^*\mathcal{O}_{X}(mK_X))=H^0(X', \mathcal{O}_{X'}(m(K_{X'}+\sum_ia_iE_i))) \\
& =H^0(X', \mathcal{O}_{X'}(m(g^*K_{X''}+\sum_{j}b_jF_j+\sum_ia_iE_i))) \\
& \subseteq H^0(X'', \mathcal{O}_{X''}(m(K_{X''}+\sum_ia_ig_*E_i)))
\end{align*}
for any sufficiently divisible $m$. Since $X''$ is a ruled surface, there must be an $i$ such that $g_*E_i$ dominates the base curve $C$. Since $X$ is log canonical, by \cite[3.30]{Kollar13} the genus of $E_i$, hence that of $C$, is at most 1. In particular $C$ is surjective onto its Jacobian. So according to the structure of $a_{X''}:X''\to {\rm Alb}(X'')$ explained in Example \ref{mamodiff} we are done.
\end{proof}
\begin{proof}[Proof of Theorem \ref{surj3d}]
Let $X\xrightarrow{a_X} Z=a_X(X)\subseteq {\rm Alb}(X)$ be the Albanese morphism of $X$. We denote the Stein factorization of $a_X$ by 
$$X\xrightarrow{p}Y\xrightarrow{q}Z.$$
We can assume that $Y$ is normal. As is shown in Lemma \ref{Albeq} below, it suffices to show that the Albanese morphism of $Y$ is surjective. We consider the following cases according to the dimension of $Y$ (or equivalently, that of $Z$). \\

\noindent\emph{Case 1.} If $\dim Y=0$, this is trivial. \\

\noindent\emph{Case 2.} If $\dim Y=1$, then $Y$ is smooth. Since $|-m(K_X+\Delta)|$ is base point free, we can choose a suitable prime divisor $D\sim_{\rm lin}-m(K_X+\Delta)$ such that $p>m$ and $D|_{X_0}$ does not contain any $F$-pure center of $(X_0,\Delta|_{X_0})$. Then by Lemma \ref{Fpureperturb} there is a positive integer $n$ such that $(X_0,\Delta|_{X_0}+\dfrac{1}{n}D|_{X_0})$ is sharply $F$-pure. Then by applying \cite[Theorem B]{PSZ13} on the projection $X_0\times\mathbb{P}H^0(X_0,\mathcal{O}_{X_0}(D|_{X_0}))\to \mathbb{P}H^0(X_0,\mathcal{O}_{X_0}(D|_{X_0}))$ we know that there is a positive integer $n$ such that $(X_0,\Delta|_{X_0}+\dfrac{1}{n}D_{X_0})$ is sharply $F$-pure for any divisor $D_{X_0}\in|D|_{X_0}|$ in a neighborhood of $D|_{X_0}$. In particular the pair
$$(X_0,\Delta|_{X_0}+\dfrac{1}{mn(\dim X_0+1)}(D_{X_0}^1+...+D_{X_0}^{\dim X_0+1}))$$
is sharply $F$-pure for any $D_{X_0}^i\in|D|_{X_0}|$ in a neighborhood of $D|_{X_0}$. So by using the techniques in \cite[Idea of Theorem 1]{Tanaka15b} we get that there is a $\mathbb{Q}$-divisor $D'\sim_{\mathbb{Q}}\dfrac{1}{m}D$ such that $(X_0,(\Delta+D')|_{X_0})$ is sharply $F$-pure. Since $X_0$ is normal it is $S_2$, $G_1$ and regular in codimension 1. Moreover we have $\displaystyle K_X+\Delta+D'\sim_{\mathbb{Q}}K_X+\Delta+\frac{1}{m}D\sim_{\mathbb{Q}}0$. So by \cite[Theorem 3.10]{Patakfalvi14} we have that $K_{X/Y}+\Delta+D'$ is nef, hence $-K_Y$ is nef. Since $Z$ generates ${\rm Alb}(X)$, $Y$ and $Z$ have to be elliptic curves and we are done. \\

\noindent\emph{Case 3.} If $\dim Y=2$, we denote a general fiber of $p$ by $F$. By assumption $F$ is a smooth curve. We observe that $\kappa(Y)\ge 0$, otherwise $Y$, hence $Z$, is covered by rational curves and this would be a contradiction to the construction that $Z\subseteq {\rm Alb}(X)$. We consider the following subcases. \\

\hangindent=8mm\emph{Subcase 1.} If $g(F)=0$, then $F\cong\mathbb{P}^1$. In this case we can choose the prime divisor $D\sim_{\rm lin}-m(K_X+\Delta)$ such that $\displaystyle (F,(\Delta+\frac{1}{m}D)|_F)$ is lc. By Theorem \ref{canonbund} there is an effective $\mathbb{Q}$-divisor $M$ on $Y$ such that 
$$0\sim_{\mathbb{Q}}K_{X}+\Delta+\frac{1}{m}D\sim_{\mathbb{Q}}p^*(K_{Y}+M),$$ 
hence $K_{Y}+M\sim_{\mathbb{Q}}0$ and in particular $\kappa(Y)\le 0$. Let $\nu:Y'\to Y$ be a minimal resolution of $Y$. If $M\ne 0$ or $Y$ has non-canonical singularity, then $\kappa(Y')<0$ and $Y'$ is birationally ruled. This is a contradiction to the fact that $Y$ has a finite map $q:Y\to Z\subseteq {\rm Alb}(X)$ onto its image. Therefore $Y$ has canonical singularities and by Lemma \ref{surjsurface} we are done. \\

\setlength\parindent{12pt}\hangindent=8mm\emph{Subcase 2.} If $g(F)\ge 1$, then after possibly doing a base change and a resolution, WLOG we can assume that both $X$ and $Y$ are smooth. Then by \cite[3.2]{CZ15} we have $\kappa(\omega_{X/Y})\ge 0$. This together with $\kappa(X)\le 0$ implies that $\kappa(Y)\le 0$. Again by the argument in \emph{Subcase 1} and Lemma \ref{surjsurface} we are done. \\

\hangindent=0mm
\noindent\emph{Case 4.} If $\dim Y=3$ then the $a_X$ is generically finite. Moreover since $X_0$ is normal it is reduced, hence $a_X$ is separable. So by \cite[Theorem 0.2]{HP15b} we are done.
\end{proof}
\section{Results in characteristic $0$}
In this section we work in characteristic $0$.
\begin{definition}
Let $(X,\Delta)$ be a pair and $a:X\dasharrow Y$ a rational map from $X$ to another variety $Y$ . We say that \emph{the non-lc locus of $(X,\Delta)$ dominates} $a(X)$ if there is a divisor $E$ of discrepancy $< -1$ that dominates $a(X)$.
\end{definition}
\begin{theorem} \label{charactnonklt}
Let $(X,\Delta)$ a pair such that $\kappa(K_X+\Delta)=0$. Suppose that the Albanese map $\mathfrak{a}_X:X\dasharrow \mathfrak{Alb}(X)$ (the Albanese morphism $a_X:X\to {\rm Alb}(X)$) of $X$ is not an algebraic fiber space. Then the non-lc locus of $(X,\Delta)$ dominates $\mathfrak{a}_X(X)$ ($a_X(X)$).
\end{theorem}
\begin{lemma}\label{obvious}
Let $f:X\to Z$ be a morphism of varieties. Let $E$ be a divisor over $X$ and $g:Z'\to Z$ a birational morphism.  Then $E$ dominates $Z$ if and only if $E$ dominates $Z'$.
\end{lemma}
\begin{proof}
Obvious.
\end{proof}
\begin{proof}[Proof of Theorem \ref{charactnonklt}]
We first deal with the case of the Albanese map. The proof for the Albanese moprhism is almost the same and will be explained at the end of the proof. Let $\mu:Y\to X$ be a log resolution of $(X,\Delta)$. We can write 
$$K_Y+\tilde{\Delta}=\mu^*(K_X+\Delta)+E_+-E_-$$
where $\tilde{\Delta}$ is the strict transform of $\Delta$, and $E_+$ and $E_-$ are effective exceptional divisors of $\mu$ with no common components. Denote $D:=\tilde{\Delta}+E_-$ and for convenience we denote the Albanese morphism of $Y$ by $a$ as well. We have $\kappa(K_Y+D)=\kappa(K_X+\Delta)=0$ (see \cite[Lemma II.3.11]{Nakayama04}). Let 
\begin{align}\label{SteinfactY}
Y\xrightarrow{g} Z\xrightarrow{h} a_Y(Y)\subseteq {\rm Alb}(Y)
\end{align}
be the Stein factorization of $a_Y$. 
\begin{lemma}\label{Albeq}
If the Albanese morphism of $Z$ is an algebraic fiber space, then so is that of $Y$.
\end{lemma}
\begin{proof}
After possibly passing to a smooth model of $Z$ we can assume that the Albanese map of $Z$ is a morphism. By universality, $h$ factors through the Albanese morphism $a_{Z}$ of $Z$, and $a_{Z}\circ g$ factors through $a_{Y}$. So we have the following diagram:
\begin{center}
\begin{tikzpicture}[scale=1.6]
\node (A) at (0,0) {$Y$};
\node (B) at (1.5,0) {${\rm Alb}(Y)$};
\node (C) at (0,1.5) {$Z$};
\node (D) at (1.5,1.5) {${\rm Alb}(Z)$};
\path[->,font=\scriptsize]
(A) edge node[above]{$a_{Y}$} (B)
(C) edge node[above]{$a_{Z}$} (D)
(C) edge node[left]{$h$} (B)
(A) edge node[left]{$g$} (C)
;
\draw[-latex] (D.260) -- node[rotate=0,xshift=10pt] {s} (B.100);
\draw[-latex] (B.70) -- node[rotate=0,xshift=-10pt] {t} (D.290);
\end{tikzpicture} 
\end{center} 
By construction we have that 
\begin{itemize}
\item $s\circ t\circ a_Z\circ g=s\circ h\circ g=s\circ a_Y=a_Z\circ g$;
\item $t\circ s\circ a_Y=t\circ a_Z\circ g=h\circ g=a_Y$.
\end{itemize}
Since $g$ is surjective, we see that $s\circ t={\rm id}_{a_{Z}(Z)}$ and $t\circ s={\rm id}_{a_{Y}(Y)}$. Since $a_{Z}(Z)$ and $a_{Y}(Y)$ generate ${\rm Alb}(Z)$ and ${\rm Alb}(Y)$ respectively we know that $s$ and $t$ are isomorphisms. Moreover since $g$ and $a_Z$ are both algebraic fiber spaces we are done.
\end{proof}
Suppose that no components of $D$ with coefficients $>1$ dominate $a_Y(Y)$, or equivalently, $Z$. By \cite[Theorem 13]{Kawamata81}, there is an \'{e}tale cover $q:Z'\to Z$ such that $Z'\cong B\times W$ where $B$ is an abelian variety, $W$ is birational to a smooth variety which is of general type and of maximal Albanese dimension. Let $Y':=Y\times_ZZ'$. We then take a resolution $\nu:V\to W$ and let $Z''$ and $Y''$ be the corresponding base changes. Let $D':=p^*D$, and $\rho:Y^{\sharp}\to Y''$ be a log resolution of $(Y'',r_*^{-1}D')$, where $r_*^{-1}D'$ is the strict transform of $D'$ on $Y''$. The situation is as follows.
\begin{center}
\begin{tikzpicture}[scale=1.6]
\node (A) at (0,0) {$Y$};
\node (B) at (1.5,0) {$Z$};
\node (D) at (0,1) {$Y'$};  
\node (E) at (1.5,1) {$Z'=B\times W$};
\node (F) at (3,1) {$W$};
\node (J) at (0,-1) {$X$};
\node (K) at (0,2) {$Y''$}; 
\node (L) at (1.5,2) {$Z''=B\times V$};
\node (M) at (3,2) {$V$};
\node (N) at (-1.5,2) {$Y^{\sharp}$};
\path[->,font=\scriptsize]
(A) edge node[above]{$g$} (B)
(D) edge node[above]{$g'$} (E)
(E) edge node[above]{$p_W$} (F)
(D) edge node[left]{$p$} (A)
(E) edge node[left]{$q$} (B)
(A) edge node[left]{$\mu$} (J)
(K) edge node[left]{$r$} (D)
(L) edge node[left]{$s$} (E)
(M) edge node[left]{$\nu$} (F)
(K) edge node[above]{$g''$} (L)
(L) edge node[above]{$p_V$} (M)
(N) edge node[above]{$\rho$} (K)
;
\end{tikzpicture} 
\end{center} 
By construction, $V$ is smooth, of general type and of maximal Albanese dimension. We define $D^{\sharp}$ via the following equality:
$$K_{Y^{\sharp}}+D^{\sharp}=(r\circ\rho)^*(K_{Y'}+D')+E,$$
where $D^{\sharp}$ and $E$ are effective and have no common components. Then $D^{\sharp}$ has snc support as well. Since $p$ is an \'{e}tale cover and $r\circ \rho$ is a birational morphism, by \cite[Lemma II.3.11]{Nakayama04} we have 
$$0=\kappa(K_Y+D)=\kappa(K_{Y^{\sharp}}+D^{\sharp}).$$ 
By Lemma \ref{obvious}, no components of $D^{\sharp}$ with coefficient $\ge 1$ dominate $Z''$. We denote the horizontal part of $D^{\sharp}$ with respect to $g''\circ\rho$ as $(D^{\sharp})^{\rm hor}$. Then the coefficients of $(D^{\sharp})^{\rm hor}$ are in $[0,1]$. By \cite[Theorem 4.2 and Remark 4.3]{Campana04} we know that  
\begin{align}\label{Campana}
\kappa(K_{Y^{\sharp}}+(D^{\sharp})^{\rm hor})\ge \kappa((K_{Y^{\sharp}}+(D^{\sharp})^{\rm hor})|_{F_{Y^{\sharp}/V}})+\kappa(V).
\end{align}
and $\kappa(V)\ge 0$ as $V$ is of general type. 

On the other hand, we have  
\begin{align*}
\kappa(K_{F_{Y^{\sharp}/Z''}}+(D^{\sharp})^{\rm hor}|_{F_{Y^{\sharp}/Z''}})=\kappa(K_{F_{Y^{\sharp}/Z''}}+D^{\sharp}|_{F_{Y^{\sharp}/Z''}})\ge 0.
\end{align*}
So by Lemma \ref{GVandKodairadim}, $\kappa((K_{Y^{\sharp}}+(D^{\sharp})^{\rm hor})|_{F_{Y^{\sharp}/V}})\ge 0$, and by \eqref{Campana} we have $\kappa(K_{Y^{\sharp}}+(D^{\sharp})^{\rm hor})\ge 0$. We also have $\kappa(K_{Y^{\sharp}}+D^{\sharp})=0$, which forces that $\kappa(K_{Y^{\sharp}}+(D^{\sharp})^{\rm hor})=0$. Again by \eqref{Campana}, we obtain $\kappa(V)=0$. But $V$ is of general type by construction, so $V$, hence $W$, has to be a point. Now we consider the composition of maps
$$B\xrightarrow{q} Z\xrightarrow{h}{\rm Alb}(Y).$$
Since $q$ and $h$ are finite and $h(Z)$ generates ${\rm Alb}(Y)$, the composite $h\circ q$ has to be an isogeny, hence $Z$ is birational to an abelian variety. Finally by Lemma \ref{Albeq} we are done.

For the case of the Albanese moprhism, we just do the Stein factorization as in \eqref{SteinfactY} for $a_X$ without taking the resolution, and the rest of the proof is the same.
\end{proof}
%
Now we are going to prove Theorem \ref{HPgen} \ref{partCH} and \ref{partch01}. We first show that the condition on the log plurigenera implies that the Albanese map and the Albanese morphism are surjective. 
\begin{proposition} \label{surjAlb}
Let $(X,\Delta)$ be a projective pair. If one of the following conditions holds:
\begin{enumerate}[label={\rm (\roman*)}]
\item $(X,\Delta)$ is lc, $h^0(X,\mathcal{O}_X(m(K_X+\Delta)))=h^0(X,\mathcal{O}_X(2m(K_X+\Delta)))=1$ for some integer $m\ge 1$. \label{surjcondlc}
\item $(X,\Delta)$ is klt, $h^0(X,\mathcal{O}_X(m(K_X+\Delta)))=1$ for some integer $m\ge 2$. \label{surjcondklt}
\end{enumerate}
then both the Albanese map $\mathfrak{a}_X:X\dasharrow \mathfrak{Alb}(X)$ and the Albanese morphism ${a}_X:X\dasharrow {\rm Alb}(X)$ of $X $ are surjective.
\end{proposition}
\begin{proof}
We first reduce Proposition \ref{surjAlb} to the case of the Albanese map. Assume that Proposition \ref{surjAlb} holds for Albanese maps. We take a log resolution for $(X,\Delta)$ and denote it by $\mu: X'\to X$. Then $\mathfrak{a}_{X'}=a_{X'}$ is a morphism from $X'$ to $\mathfrak{Alb}(X)$. We can write
$$K_{X'}+\mu_*^{-1}\Delta=\mu^*(K_X+\Delta)+E-F,$$
where $E$ and $F$ are effective exceptional $\mathbb{Q}$-divisors on $X'$ which have no common component. Clearly, for every $m\ge 1$, $(X',\mu_*^{-1}\Delta+F+\frac{1}{m}(\lceil mE\rceil-mE))$ has the same type of singularity as $(S,\Delta)$, and
\begin{align*}
& h^0(X,\mathcal{O}_X(m(K_X+\Delta)))=h^0(X',\mathcal{O}_{X'}(\mu^*(m(K_X+\Delta))+\lceil mE\rceil)) \\
=& h^0(X',\mathcal{O}_{X'}(m(K_{X'}+\mu_*^{-1}\Delta+F+\frac{1}{m}(\lceil mE\rceil-mE)))).
\end{align*}
So if $\mathfrak{a}_{X'}$, hence $\mathfrak{a}_{X}$, is surjective, then by Remark \ref{birsurAlb}, $a_X=\alpha\circ \mathfrak{a}_{X}$ is surjective.

Next, we prove part \ref{surjcondlc}. By selecting a proper birational model for $(X,\Delta)$, we can assume that $(X,\Delta)$ is log smooth, and in particular the Albanese map of $X$ is a morphism. We first show that $\mathcal{O}_X$ is an isolated point of $V^0(\mathcal{O}_X(m(K_X+\Delta)))$. The proof of this fact is the same as \cite[Proposition 2.1]{EL97} (cf. \cite[Theorem 17.10]{Kollar95}). Since $h^0(X,\mathcal{O}_X(m(K_X+\Delta)))=1$, $\mathcal{O}_X$ lies in $V^0(\mathcal{O}_X(m(K_X+\Delta)))$. Suppose that it is not an isolated point. By \cite[Theorem 1.3]{Shibata15}, $V^0(\mathcal{O}_X(m(K_X+\Delta)))$ contains a subgroup $S$ of positive dimension, and therefore for each $P\in S$ the image of 
\begin{align} \label{hahael}
\begin{split}
& H^0(X,\mathcal{O}_X(m(K_X+\Delta))\otimes P)\otimes H^0(X,\mathcal{O}_X(m(K_X+\Delta))\otimes P^{\vee}) \\ 
& \to H^0(X,\mathcal{O}_X(2m(K_X+\Delta)))
\end{split}
\end{align}
is non-zero. Since a given divisor in $|2m(K_X+\Delta)|$ has only finitely may irreducible components, it follows that as $y$ varies over the positive dimesional torus $S$, the image of \eqref{hahael} must vary as well. Therefore $h^0(X,\mathcal{O}_X(2m(K_X+\Delta)))>1$, a contradiction. 

By Theorem \ref{ps14}, $a_{X,*}\mathcal{O}_X(m(K_X+\Delta))$ is a GV-sheaf. For convenience, for the rest of the proof we denote $A:={\rm Alb(X)}$, $g:=\dim A$, and denote $a_{X,*}\mathcal{O}_X(m(K_X+\Delta))$ by $\mathcal{F}$. $\mathcal{O}_X$ being an isolated point of $V^0(\mathcal{O}_X(m(K_X+\Delta)))$ implies that there is a summand of $R^0\hat{S}(\mathcal{F})$ supported only at the origin $\hat{0}$ of $\hat{A}$. We have that 
\begin{align*}
R\hat{S}\circ D_A(\mathcal{F})=R^0\hat{S}\circ D_A(\mathcal{F})=(-1_{\hat{A}})^*((D_{\hat{A}}\circ R^0\hat{S}(\mathcal{F}))[-g]),
\end{align*}
where the first equality is by \cite[Theorem 1.2]{Hacon04} and \cite[Theorem A]{PP11}, and the second equality is by \cite[(3.8)]{Mukai81}. So the sheaf $R\hat{S}(D_A(\mathcal{F}))$ has a summand which is supported only at the origin $\hat{0}$. Then by \cite[Example 2.9]{Mukai81} (note also that $\omega_A\cong\mathcal{O}_A$), there is a summand of $\mathcal{F}^{\vee}$ which is a unipotent vector bundle. Therefore by the definition of $\mathcal{F}$, $a_X$ is surjective. 

Now we prove part \ref{surjcondklt}. By part \ref{surjcondlc} (or by Theorem \ref{charactnonklt}), we can assume that $\kappa(K_X+\Delta)>0$. We construct a commutative diagram following \cite[Theorem 1]{CH02}. Let $f:X\to V$ be a birational model of the Iitaka fibration for $K_X+\Delta$. As explained in the last paragraph, we can assume that $(X,\Delta)$ is log smooth. Since $X$ is smooth now, the Albanese map of $X$ is a morphism which we denote by $a_X:X\to {\rm Alb}(X)$. We construct the following diagram:
\begin{center}
\begin{tikzpicture}[scale=1.6]
\node (A) at (0.5,0) {$V$};
\node (C) at (2,0) {$W$};
\node (D) at (3,0) {$S$};  
\node (E) at (0.5,1) {$X$}; 
\node (G) at (2,1) {$Z$};
\node (H) at (3,1) {${\rm Alb}(X)$};
\path[->, font=\scriptsize]
(A) edge node[below]{$h$} (C)
(C) edge node[above]{$j$} (D)
(E) edge node[above]{} (G)
(G) edge node[above]{$i$} (H)
(E) edge node[left]{$f$} (A)
(H) edge node[right]{$p$} (D)
(E) edge node[xshift=4pt, yshift=4pt]{$g$} (C);
\draw[->] (E) edge[out=30,in=150] node[pos=0.5,yshift=7pt] {$a_X$} (H);
\end{tikzpicture} 
\end{center}  
Here $S$ and $p$ are defined in the following way. Denote the general fiber of $f$ by $X_f$. We then have that $(X_f,\Delta|_{X_f})$ is lc (for a proof see \cite[Lemma 2.7]{Wang16} where klt case is shown, and lc case is the same). So by Theorem \ref{charactnonklt}, the Albanese morphism of $X_f$ is an algebraic fiber space. In particular, $a_X(X_f)$ is some translation of an abelian subvariety of ${\rm Alb}(X)$, which we denote by $K$. Then we define $S:={\rm Alb}(X)/K$ and define $p$ to be the quotient morphism. Now we have a rational map from $V$ to $S$, and we can arrange the birational model of the Iitaka fibration of $X$ so that we actually have a morphism from $V$ to $S$. Define $Z:=a_X(X)$, $W:=p(a_X(X))$, $i$ and $j$ the natural inclusions. 
Again, we can arrange $f$ birationally such that $X$ and $V$ admit morphisms to $Z$ and $W$ respectively.

For the rest of the proof we assume that $m=2$. The proof proceeds analogously for any $m\ge 2$. 

We fix a positive integer $m_1$ such that $\dim(|m_1(K_X+\Delta)|)>0$. Replacing $X$ by an appropriate birational model, we may assume that the linear system $|m_1(K_X+\Delta)|$ has base locus of pure codimension 1. For general member $B_1\in |m_1(K_X+\Delta)|$ we write 
\begin{align*}
& B_1=D_1+\sum_{i=1}^ra_iE_i 
\end{align*}
where $D_1$ denotes the free part and $\sum_{i=1}^ra_iE_i $ denotes the base divisor. 
Recall that the Iitaka fibration of $X$ with respect to $K_X+\Delta$ is constructed as the morphism induced by the base-point-free part $|M|$ of $|m(K_X+\Delta)|$ for some $m$. In particular, there is an integer $m_2>0$ and an ample divisor $A$ on $W$ such that $|m_2(K_X+\Delta)-g^*A|$ is not empty.
We fix such an $m_2$ and choose a $B_2\in |m_2(K_X+\Delta)-g^*A|$. We can similarly write
$$B_2=D_2+\sum_{j=1}^sb_jF_j$$
where $D_2$ is the free part and $\{F_j\}$ are the base divisors. We can move $D_1$ and $D_2$ such that $D_{1,2}$ do not coincide with $E_i$, $F_j$ or any component of $\Delta$. By choosing a birational model of $X$ properly, we may assume that ${\rm Supp}(B_1)\bigcup{\rm Supp}(B_2)\bigcup{\rm Supp}(\Delta)$ is simple normal crossing. 

Consider the divisor
$$B:=kB_1+B_2\in |(km_1+m_2)(K_X+\Delta)-g^*A|$$
We want to show that for $k\gg 0$, for any component $P$ in $\lfloor\frac{B}{km_1+m_2}+\Delta\rfloor$,
\begin{align} \label{inequalwithoutfloor}
{\rm mult}_{P}(\frac{B}{km_1+m_2}+\Delta)< {\rm mult}_{P}(\frac{2}{m_1}\sum_{i=1}^ra_iE_i)+1\le {\rm mult}_{P}F+1.
\end{align}
The equality on the right always holds, and this is easily seen as $\frac{m_1}{2}|2(K_X+\Delta)|\subseteq|m_1(K_X+\Delta)|$. For the equality on the left, we proceed by simply observing the multiplicities of each component. We have $B=kD_1+D_2+\sum_{i=1}^rka_iE_i+\sum_{j=1}^sb_jF_j$, and obviously $\lfloor \frac{k}{km_1+m_2}\rfloor=\lfloor \frac{1}{km_1+m_2}\rfloor=0$. Let $\Delta=\sum_ld_l\Delta_l$ where $\Delta_l$ are the irreducible components of $\Delta$ and $d_l\in [0,1)$. The worst case now is that $P$ is a common component of $\{E_i\}$, $\{F_j\}$ and $\Delta$. Then its multiplicity is 
$$\dfrac{ka_i+b_j}{km_1+m_2}+d_l,$$
which is easily seen as strictly less that $\dfrac{2a_i}{m_1}+1$ for $k\gg 0$. So \eqref{inequalwithoutfloor} holds, hence 
$$\lfloor \frac{B}{km_1+m_2}+\Delta\rfloor\preceq F.$$
Define 
$$L:=K_X+2\Delta-\lfloor \frac{B}{km_1+m_2}+\Delta\rfloor.$$ 
We have that 
$$\displaystyle L\equiv_{\rm num}\frac{1}{km_1+m_2}\pi^*H+\{\frac{B}{km_1+m_2}+\Delta\}$$ 
By constuction it is easy to see that for $k\gg 0$, $(X, \{\frac{B}{km_1+m_2}+\Delta\})$ is klt. So $L$ is numerically equivalent to the sum of the pull back of a nef and big $\mathbb{Q}$-divisor on $X$ and a klt $\mathbb{Q}$-divisor on $W$.

Therefore we have that $|2(K_X+\Delta)|=|K_X+L|+\lfloor \frac{B}{km_1+m_2}+\Delta\rfloor$, and $h^0(X,\mathcal{O}_X(K_X+L))=1$. So by \cite[Corollary 10.15]{Kollar95}, $h^i(W,g_*\mathcal{O}_X(K_X+L)\otimes P)=0$ for all $i>0$ and $P\in {\rm Pic}^0(W)$. It follows that 
\begin{align*}
& h^0(W,g_*\mathcal{O}_X(K_X+L)\otimes P)=\chi(W,g_*\mathcal{O}_X(K_X+L)\otimes P) \\
& =\chi(W,g_*\mathcal{O}_X(K_X+L))=h^0(W,g_*\mathcal{O}_X(K_X+L))=1
\end{align*}
for all $P\in {\rm Pic}^0(S)$. We also have that $g_*\mathcal{O}_X(K_X+L)$ is a torsion free coherent sheaf on $X$. So by \cite[Proposition 1.2.1]{CH02}, $g_*\mathcal{O}_X(K_X+L)$ is supported on an abelian subvariety $S'$ of $S$. Since $g_*\mathcal{O}_X(K_X+L)$ is a torsion free and the image of $X$ generates $S$, we see that $S'=S$. Hence $\mathfrak{a}_X$ is surjective.
\end{proof}
\begin{theorem} \label{charactCHlog}
Let $(X,\Delta)$ be a projective pair. Assume one of the following conditions holds:
\begin{enumerate}[label={\rm (\arabic*)}]
\item $(X,\Delta)$ is lc, $h^0(X,\mathcal{O}_X(m(K_X+\Delta)))=h^0(X,\mathcal{O}_X(2m(K_X+\Delta)))=1$ for some integer $m\ge 1$. \label{surjcondlc'}
\item $(X,\Delta)$ is klt, $h^0(X,\mathcal{O}_X(m(K_X+\Delta)))=1$ for some integer $m\ge 2$. \label{surjcondklt'}
\end{enumerate}
Moreover, assume that either $\dim(X)=\dim(\mathfrak{Alb}(X))$ or $\dim(X)=\dim({\rm Alb}(X))$, then $X$ is birational to an abelian variety.
\end{theorem}
\begin{proof}[Proof of Theorem \ref{charactCHlog}]
We only show the case of $\dim(X)=\dim(\mathfrak{Alb}(X))$, and the proof for the Albanese morphism is the same. After possibly taking a log resolution for $(X,\Delta)$ we can assume that $(X,\Delta)$ is log smooth. In particular, $q(X)=\dim (X)$. By Proposition \ref{surjAlb}, $\mathfrak{a}_X$ is surjective. This combining with the condition $\dim({\rm Alb}(X))=\dim(X)$ implies that $\mathfrak{a}_X$ is generically finite onto $\mathfrak{Alb}(X)$. 

Next by \cite[Lemma 3.1]{CH01} (by letting $Y=A=C=\mathfrak{Alb}(X)$ there) we get that 
$$1=P_{m'}(\mathfrak{Alb}(X))\le P_{m'}(X)\le h^0(X, \mathcal{O}_X(m'(K_X+\Delta)))=1.$$ 
Here $m'=2m$ in \ref{surjcondlc'} and $m'=m$ in \ref{surjcondklt'}. This forces that $P_{m'}(X)=1$, so by \cite[Corollary 2]{CH02} we are done.
\end{proof}
\bibliographystyle{alpha}
\bibliography{P}  
\end{document}